\documentclass[reqno]{amsart}
\usepackage{graphicx}
\usepackage{caption}
\usepackage{subcaption}
\usepackage{color}
\usepackage{pdfpages}

\usepackage{amssymb}
\usepackage{mathrsfs}
\usepackage{enumitem}
\usepackage{hyperref}
\usepackage{isomath}

\begin{document}

\newtheorem{thm}{Theorem}[section]
\newtheorem*{thm1}{Theorem}
\newtheorem{lem}[thm]{Lemma}
\newtheorem{cor}[thm]{Corollary}
\newtheorem{add}[thm]{Addendum}
\newtheorem{prop}[thm]{Proposition}
\theoremstyle{definition}
\newtheorem{defn}[thm]{Definition}
\newtheorem{claim}[thm]{Claim}
\newtheorem*{mainthm}{Theorem 1.1}
\theoremstyle{remark}
\newtheorem{rmk}[thm]{Remark}
\newtheorem{ex}[thm]{Example}
\newtheorem{conj}[thm]{Conjecture}

\newcommand{\CC}{\mathbb{C}}
\newcommand{\RR}{\mathbb{R}}
\newcommand{\DD}{\mathbb{D}}

\newcommand{\ZZ}{\mathbb{Z}}
\newcommand{\QQ}{\mathbb{Q}}
\newcommand{\NN}{\mathbb{N}}
\newcommand{\FF}{\mathbb{F}}
\newcommand{\PP}{\mathbb{P}}
\newcommand{\CmodTwoPiIZ}{{\mathbf C}/2\pi i {\mathbf Z}}
\newcommand{\Cnozero}{{\mathbf C}\backslash \{0\}}
\newcommand{\Cinfty}{{\mathbf C}_{\infty}}
\newcommand{\RPnminustwo}{\mathbb{RP}^{n-2}}

\newcommand{\SLtwoC}{\mathrm{SL}(2,{\mathbb C})}
\newcommand{\SLtwoR}{\mathrm{SL}(2,{\mathbb R})}
\newcommand{\PSLtwoC}{\mathrm{PSL}(2,{\mathbb C})}
\newcommand{\PSLtwoR}{\mathrm{PSL}(2,{\mathbb R})}
\newcommand{\SLtwoZ}{\mathrm{SL}(2,{\mathbb Z})}
\newcommand{\PSLtwoZ}{\mathrm{PSL}(2,{\mathbb Z})}
\newcommand{\PGLtwoQ}{\mathrm{PGL}(2,{\mathbb Q})}

\newcommand{\A}{{\mathcal A}}
\newcommand{\B}{{\mathcal B}}
\newcommand{\C}{{\mathcal X}}
\newcommand{\D}{{\mathcal D}}
\newcommand{\E}{{\mathcal E}}
\newcommand{\F}{{\mathcal F}}

\newcommand{\OO}{{\mathcal O}}

\newcommand{\MCG}{\mathcal{MCG}}
\newcommand{\EE}{\mathbb{E}^2}
\newcommand{\HH}{\mathbb{H}^2}
\newcommand{\HHH}{\mathbb{H}^3}
\newcommand{\TTT}{\mathcal{T}}
\newcommand{\QQQ}{\mathcal{Q}}
\newcommand{\SSS}{\mathcal{S}}
\newcommand{\tr}{{\hbox{tr}\,}}
\newcommand{\Hom}{\mathrm{Hom}}
\newcommand{\Aut}{\mathrm{Aut}}
\newcommand{\Inn}{\mathrm{Inn}}
\newcommand{\Out}{\mathrm{Out}}
\newcommand{\SL}{\mathrm{SL}}
\newcommand{\BQ}{\rm{BQ}}
\newcommand{\Id}{\rm{Id}}
\newcommand{\Width}{\rm{Width}}

\newcommand{\setn}{{[n]}}
\newcommand{\powern}{{P(n)}}
\newcommand{\nck}{{C(n,k)}}

\newcommand{\hatI}{{\hat{I}}}
\newcommand{\TkDelta}{{T^{|k|}(\Delta)}}
\newcommand{\vecDelta}{{\vec{\Delta}_{\phi}}}

\newcommand{\Tabstwo}{{T^{|2|}(\Delta)}}
\newcommand{\Tnminusone}{{T^{|n-1|}(\Delta)}}
\newcommand{\Hur}{{\mathcal{H}}}
\newcommand{\JW}{{\hbox{JW}}}
\newcommand{\wt}{{\hbox{wt}}}
\newcommand{\Deltaone}{{\Delta^{(1)}}}
\newcommand{\Deltatwo}{{\Delta^{(2)}}}
\newcommand{\Deltathree}{{\Delta^{(3)}}}
\newcommand{\Deltan}{{\Delta^{(n)}}}

\newenvironment{pf}{\noindent {\it Proof.}\quad}{\square \vskip 10pt}

\title[Pseudomodular groups]{Hyperbolic jigsaws and families  of pseudomodular groups I}
\author[B. Lou, S.P. Tan and A.D. Vo ]{Beicheng Lou, Ser Peow Tan and  Anh Duc Vo }

\address{National University of Singapore \\
	Singapore 119076} \email{lbc45123@hotmail.com}
\address{Department of Mathematics \\
	National University of Singapore \\
	Singapore 119076} \email{mattansp@nus.edu.sg}
\address{National University of Singapore \\
	Singapore 119076} \email{voducht8728@gmail.com}

\subjclass[2000]{}

\thanks{The second author is partially supported by the National University
	of Singapore academic research grant R-146-000-235-114.}   

\date{\today}

%
%

 \begin{abstract}
 We show that there are infinitely many commensurability classes of pseudomodular groups, thus answering a question raised by Long and Reid. These are Fuchsian groups whose cusp set is all of the rationals but which are not commensurable to the modular group. We do this by introducing a general construction for the fundamental domains of Fuchsian groups obtained by gluing together marked ideal triangular tiles, which we call hyperbolic jigsaw groups. 

 \end{abstract}

 \maketitle
 \tableofcontents

 \vspace{10pt}
\section{Introduction}
We first recall the  definition of a pseudomodular group  which is the main motivation for this paper. Recall that the cusp set of a Fuchsian group $\Gamma \le \PSLtwoR$ is the set of all $x \in  \RR \cup \{\infty\}=\partial \HH$ such that the stabilizer of $x$ in $\Gamma$ is generated by a parabolic element.

\begin{defn}
	Let $\Gamma \le \PGLtwoQ$ be a discrete group such that $\HH/\Gamma$ is a complete hyperbolic surface of finite area. If $\Gamma$ is not commensurable with $\PSLtwoZ$, and $\Gamma$ has cusp set precisely $\hat \QQ=\QQ \cup \{\infty\}$, then $\HH/\Gamma$ is called a pseudomodular surface and $\Gamma$  a pseudomodular group.
	
\end{defn}

These were introduced by Long and Reid in \cite{LR}, the main surprise, and remarkable fact is that such  groups  exist. Clearly, the property of being pseudomodular carries over to the commensurability class, so  we can consider commensurability classes of pseuodomodular groups. Four different commensurability classes of pseudomodular groups were found in \cite{LR}, although a fifth, $\triangle(5/13,4)$, was also known to the authors (private communication), and found independently by Ayaka \cite{Ay} and Junwei Tan \cite{TJW}, using ideas of \cite{LR}.  Long and Reid asked if there were (in)finitely many commensurability classes of pseudomodular groups (  Question 2 of \S 6 in \cite{LR}). 

\medskip

So far,  not much progress has been made towards this question. Some partial answers were provided by Proskin \cite{Pr}, and a couple of new pseudomodular groups have been found by Ayaka \cite{Ay}.  On the other hand, using methods from number theory, Fithian \cite{Fi} was able to give conditions for a group to be non-pseudomodular. This however, does not answer the question if there are infinitely many commensurability classes of pseudomodular groups. 

\medskip
Our main result is the following answer to Long and Reid's question:

\begin{thm}\label{thm:Main}
	There exists infinitely many commensurability classes of pseudomodular groups and surfaces.
\end{thm}
Theorem \ref{thm:Main}  follows from either Theorem \ref{thm:S12} or \ref{thm:S13}.  Our idea is a general way to construct subgroups of $\PGLtwoQ$ by constructing their fundamental domains by gluing together tiles from a finite collection $\mathcal S$  of marked hyperbolic ideal triangle tiles satisfying certain balancing and rationality conditions, and with  matching conditions for the gluing. We call the polygons  obtained from this construction  $\mathcal S$-jigsaws, and the corresponding groups and surfaces ${\mathcal S}$-jigsaw groups and surfaces. In Theorem \ref{thm:S12}, we show  that for the collection $\mathcal S=\{\Delta(1,1,1),\Delta(1,1/2,2)\}$, (see \S \ref{s:jigsaw} for the definition of $\Delta(1,1,1)$ and $\Delta(1,1/2,2)$),  all hyperbolic $\mathcal S$-jigsaw groups and surfaces  are pseudomodular. Furthermore, there are infinitely many commensurability classes of such groups. On the other hand, in Theorem \ref{thm:S13}, we show that for $\SSS=\{\Delta(1,1,1),\Delta(1,1/3,3)\}$, arithmetic, pseudomodular, and non-maximally cusped groups all can occur.  There are also infinitely many (commensurability classes of) $\SSS$-jigsaw groups which are pseudomodular, and infinitely many (commensurability classes of) $\SSS$-jigsaw groups which are not pseudomodular.
We give exact criteria for determining when such a group is arithmetic,  pseudomodular or neither.
\medskip

 In some sense, the groups constructed in this way, in particular the groups arising  from the jigsaws with two tiles, $J= \Delta(1,1,1)\cup \Delta(1,1/2,2)$ or $J'=\Delta(1,1,1)\cup \Delta(1,1/3,3)$ are closest to being pseudomodular. We will explore  properties of these groups, as well as general integral jigsaw groups in a future paper \cite{LTV}. For example, we will describe an associated pseudo-euclidean algorithm and generalized continued fraction for some of these groups.  We will also show that for any integral hyperbolic jigsaw set $\SSS$, there are infinitely many commensurability classes of pseudomodular $\SSS$-jigsaw groups and  explore the question of arithmeticity for general integral jigsaw groups.

The rest of this paper is organized as follows. In \S \ref{s:jigsaw}, we give the definitions and some basic properties of ${\mathcal S}$-jigsaws, their associated surfaces and groups, and state the main results. In \S \ref{s:killer}, we describe Long and Reid's idea of covering a fundamental interval for the group by killer intervals to show that the cusp set is $\QQ \cup \{\infty\}$. In \S \ref{s:basic}, we describe integral hyperbolic jigsaw sets and groups and give some of their basic properties. In \S \ref{s:onetwojigsaws} and \S \ref{s:onethreejigsaws}, we specialize these results to $\SSS(1,2)$ and $\SSS(1,3)$ jigsaws and study the cusp sets of these jigsaw groups. In \S \ref{s:non-arith}, we explore criteria for non-arithmeticity and show that all $\SSS(1,2)$ jigsaw groups are non-arithmetic, hence pseudomodular. We also give exact criteria for a $\SSS(1,3)$ jigsaw group to be arithmetic. In \S \ref{s:commensurabilty}, we look at criteria for two non-arithmetic groups to be non-commensurable and use this to construct infinitely many commensurability classes of $\SSS(1,2)$ and $\SSS(1,3)$ pseudomodular groups, and infinitely many commensurability classes of non-pseudomodular $\SSS(1,3)$ jigsaw groups. Finally, in \S \ref{s:conclusion} we put together the results of the previous sections to prove Theorems \ref{thm:S12} and \ref{thm:S13}, and hence the main results of this paper and conclude with some open questions.

\medskip

 \noindent {\it Acknowledgements}. We are grateful to Chris Leininger, Alan Reid, Darren Long, Pradthana Jaipong, Mong-Lung Lang, Yasushi Yamashita and Junwei Tan for their interest in this paper and  for helpful conversations and comments. We are particularly indebted to Chris Leininger for supplying the argument showing the non-commensurability of certain groups using the maximal horocycles, which greatly improves our earlier argument.

\section{Hyperbolic jigsaws and  main results}\label{s:jigsaw}
We will always use the upper half space model for the hyperbolic plane $\HH$, with group of orientation preserving isometries $\PSLtwoR$. We also always denote translation to the right by $1$ by
\[T:=\left(
\begin{array}{cc}
1& 1\\
0& 1\\
\end{array}
\right).\] 
Let
 $\Delta$ be a positively oriented ideal triangle with vertices $v_1,v_2,v_3$ and sides $s_1=[v_1,v_2]$, $s_2=[v_2,v_3]$ and $s_3=[v_3,v_1]$ and let $x_i \in s_i$, $i=1,2,3$ be marked points on the sides. 
Then $\Delta$, together with the points $x_1,x_2,x_3$ is called a marked triangle, two such triangles are the same if they are isometric, as marked triangles. Let $p_i$ be the point on $s_i$ which is the projection of the opposite vertex of $\Delta$ to $s_i$. Let $d_i$ be the signed distance from $p_i$ to $x_i$ along $s_i$, and  let $k_i=e^{2d_i}$. The marked triangle, up to isometry, is then determined by the triple $(k_1,k_2,k_3) \in \RR_{+}^3$, and is denoted by $\Delta(k_1,k_2,k_3)$.  Compare figure \ref{fig:markedtriangle}. We use $\Delta_{[v_1,v_2,v_3]}$ to denote the (unmarked) ideal triangle  in $\HH$ with vertices $v_1,v_2,v_3$.

\begin{figure}[hbt]\centering{\includegraphics[height=6cm]{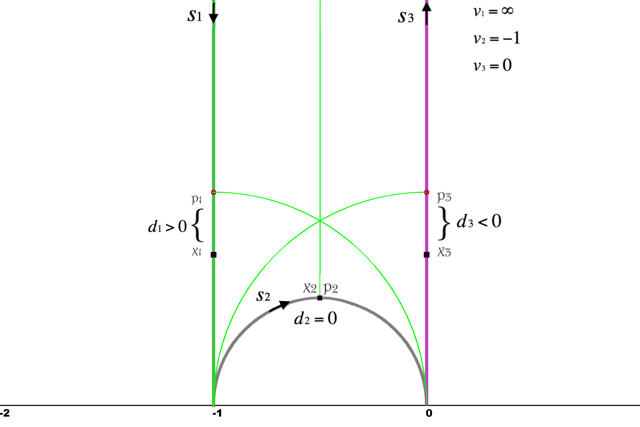}}
	\caption{A marked triangle where $v_1=\infty$, $v_2=-1$, $v_3=0$, $d_1>0$, $d_2=0$ and $d_3<0$ ($k_1>1,k_2=1,k_3<1$).}  
	\label{fig:markedtriangle} 
\end{figure}

Up to cyclic permutation, $(k_1,k_2,k_3)$ determines the same marked triangle, that is,
\[\Delta(k_1,k_2,k_3)=\Delta(k_2,k_3,k_1)=\Delta(k_3,k_1,k_2).\] 
 Note however that in general, $\Delta(k_1,k_2,k_3)$ is not isometric to $\Delta(k_2,k_1,k_3)$. 
 
 We say that  $\Delta(k_1,k_2,k_3)$ is in standard position if $v_1=\infty$, $v_2=-1$, $v_3=0$, in which case we have 
\begin{equation}\label{eqn:fixedpoints}
x_1=-1+\frac{i}{\sqrt{k_1}}, \quad x_2=\frac{-1+\sqrt{k_2}i}{1+k_2}, \quad x_3=\sqrt{k_3}i.
\end{equation}

Fix an embedding of $\Delta(k_1,k_2,k_3)$ into $\HH$ and let $\imath_j$, $j=1,2,3$ be the $\pi$-rotation about $x_j$. Then $\imath_1, \imath_2, \imath_3$ generates a discrete subgroup, denoted by  $\Gamma(k_1,k_2,k_3) \le \PSLtwoR$. It is easy to see that the corresponding surface is complete ($\imath_1\imath_2\imath_3$ is parabolic) if and only if $k_1k_2k_3=1$.  Furthermore, if $k_j \in \QQ$ for $j=1,2,3$, and the triangle is embedded so the vertices $v_j \in \QQ \cup \{\infty\} \subset \partial \HH$, then $\Gamma(k_1,k_2,k_3) <\PGLtwoQ$. We have:

\begin{defn}(Rational and balanced triangles and Weierstrass groups)
The marked triangle $\Delta(k_1,k_2,k_3)$ is balanced if $k_1k_2k_3=1$ and rational if $k_1, k_2, k_3 \in \QQ$. $\Gamma(k_1,k_2,k_3)$, defined up to conjugation by $\PSLtwoR$, is the (Weierstrass) group generated by the involutions $\imath_1, \imath_2, \imath_3$ about the marked points $x_1,x_2$, and $x_3$ respectively. It is the index two supergroup of the torus group where $x_1,x_2$ and $x_3$ correspond to the Weierstrass points on the torus.

\end{defn} 
For example, if $\Delta(k_1,k_2,k_3)$ is in standard position, then $\imath_1$, $\imath_2$, $\imath_3$ can be represented by the matrices (as elements of $\PSLtwoR$)
\begin{equation}\label{eqn:involutions}
 \frac{1}{\sqrt{k_1}}\left(
\begin{array}{cc}
k_1 & 1+k_1 \\
-k_1 & -k_1 \\
\end{array}
\right), \quad
\frac{1}{\sqrt{k_2}} \left(
\begin{array}{cc}
1 & 1 \\
-(k_2+1) & -1 \\
\end{array}
\right), \quad
\frac{1}{\sqrt{k_3}} \left(
\begin{array}{cc}
0 & k_3 \\
-1 & 0 \\
\end{array}
\right)\end{equation}
respectively.
We remark that in the notation of Long and Reid in \cite{LR}, the group $\Gamma(k_1,k_2,k_3)$ corresponds to the index two supergroup of the torus group $\Delta(u^2,2\tau)$ where \begin{equation}
u^2=k_3, \qquad \tau=1/k_1+1+k_3.
\end{equation}
In particular, the group $\Gamma(1,1/n, n )$ in standard position corresponds to the group $\Delta(n, n+2)$ in their notation.

\begin{defn}(Matching sides)
	Let $\Delta=\Delta(k_1,k_2,k_3)$, $\Delta'=\Delta(k_1',k_2',k_3')$ be two marked triangles, with sides $s_1,s_2,s_3$ and $s_1', s_2',s_3'$, and marked points $x_1,x_2,x_3$ and $x_1',x_2',x_3'$ respectively. We say that $s_i$ and $s_j'$ match if $k_i=k_j'$. 
\end{defn}

Geometrically,  this means that if we glue $\Delta$ to $\Delta'$ by gluing $s_i$ to $s_j'$, matching  $x_i$ to $x_j'$, then the $\pi$ rotation $\imath_i$ about $x_i=x_j'$ interchanges the unmarked triangles $\Delta$ and $\Delta'$, see figure \ref{fig:gluingtriangles}.

\begin{figure}[hbt]\centering{\includegraphics[height=6cm]{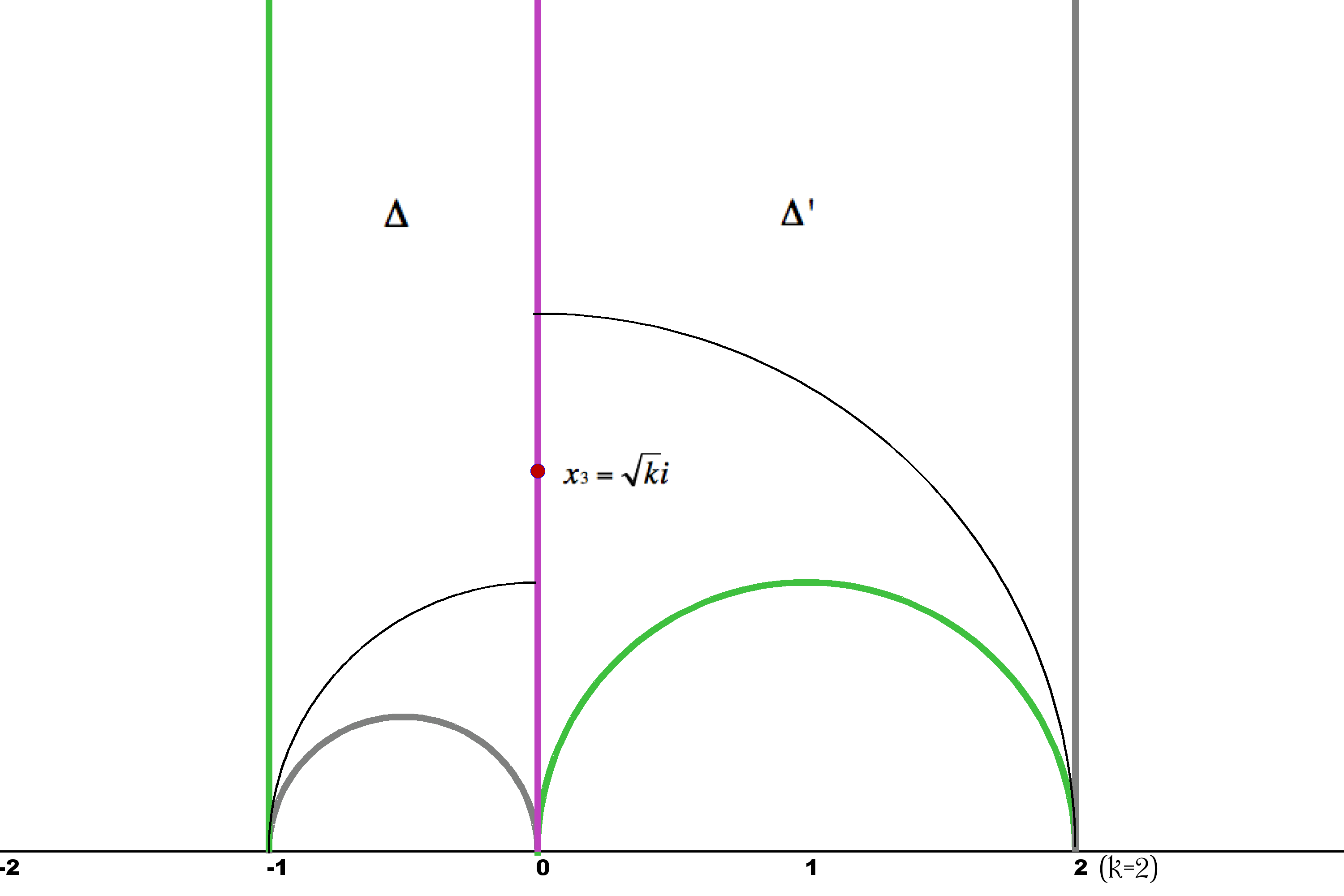}}
	\caption{Gluing two marked triangles $\Delta$ and $\Delta'$ along matching side with parameter $k=2$.}  
	\label{fig:gluingtriangles} 
\end{figure}

In terms of shear coordinates, this means that for the ideal quadrilateral formed by gluing $\Delta$ to $\Delta'$, the shear coordinate along the diagonal formed from the identified side is $\log k_i$. 

In this way,  we may glue marked triangles along matching sides, obtaining a triangulated ideal polygon unique up to isometry, with marked points on the interior and exterior sides.

\medskip

We are now ready to define the objects needed to assemble a hyperbolic jigsaw and to state our main theorem.

\medskip

\begin{defn} ~~
	
	\begin{enumerate}
		\item (Jigsaw tiles) A jigsaw tile is a {\it rational, balanced, marked ideal triangle} $\Delta(k_1,k_2,k_3)$. That is,  $k_1,k_2,k_3 \in \QQ$ and $k_1k_2k_3=1$.
		\item (Jigsaw sets) A jigsaw set is a finite (ordered) collection $\mathcal S:=\{\Delta_1, \ldots, \Delta_s\}$ of distinct jigsaw tiles such that every tile $\Delta_i \in \mathcal S$ has a side which matches with the side of some  other tile $\Delta_j\in \mathcal S$. 
		\item (Jigsaws) For a jigsaw set $\mathcal S$, an $\mathcal S$-jigsaw $J$ is an ideal polygon obtained by gluing  a finite number of tiles from $\mathcal S$ along matching sides, such that every $\Delta_j \in \SSS$ occurs at least once in $J$. As a convention, we will usually normalize one of the $\Delta_1$ tiles of $J$ to be in standard position with vertices at $\infty, -1$ and $0$, so that all other vertices of $J$ are rational. 
		\item (Size and signature of a jigsaw) The size of a jigsaw $J$, denoted by $|J|$ is the number of tiles in it. The signature $sgn(J)$ of $J$ is $(m_1,m_2, \ldots, m_s)$,  where $\mathcal S=\{\Delta_1, \ldots, \Delta_s\}$, and $m_j \in \NN$ denotes the number of $\Delta_j$ tiles in $J$. We have $|J|=\sum_{1}^s m_j$.
		\item (Jigsaw group and surface) The Fuchsian group associated to $J$, denoted by $\Gamma_J$, is the subgroup of $\PGLtwoQ$ generated by the $\pi$-rotations about the marked points on the (exterior) sides of $J$. The hyperbolic surface associated to $J$ is $S_J:=\HH/\Gamma_J$. (By our convention of fixing some $\Delta_1$ in standard position, $\Gamma_J$ is a zonal Fuchsian group, that is, contains a parabolic which fixes $\infty$).

		
		\item (Jigsaw tessellation) The tessellation $\QQQ_J$ of $\HH$ is the tessellation of $\HH$ obtained by the translates by $\Gamma_J$ of $J$ .
		\item (Jigsaw triangulation) The triangulation $\TTT_J$ of $\HH$  is the triangulation of $\HH$ obtained from the translates by $\Gamma_J$ of the triangulation of $J$ by its tiles. There is a labeling of the triangles of $\TTT_J$ by elements of the set $\{\Delta_1, \Delta_2, \ldots, \Delta_n \}$.
		\item   (Exterior and interior sides) The set of sides of $\QQQ_J$ is denoted by $\mathcal E_J$,   the set of sides of $\TTT_J$ is denoted by $\mathcal F_J$. Clearly, $\mathcal E_J \subset \mathcal F_J$. If $e \in \mathcal E_J$, it is called  an {\it exterior side}, if $e \in \mathcal F_J \setminus \mathcal E_J$, it is called an {\it interior side}. There is a labeling of the sides of $\TTT_J$ by the parameters $k_i$.
		\item (Cutting sequences of closed geodesics) The $V$-cutting sequence  of an oriented closed geodesic $\gamma$ on the surface $S_J$ is the  (periodic) sequence \\  $\overline{\Delta_{i_1},\Delta_{i_2}, \ldots, \Delta_{i_t}}$, or simply $\overline{i_1,i_2, \dots, i_t}$,  induced from the labeling of the triangles. The $W$-cutting sequence of $\gamma$ is $\overline{k_1,k_2,\ldots, k_t}$, where $k_i$ lies in the set of parameters of the sides of $\Delta_i \in \SSS$, induced from the labeling of the sides.
	\end{enumerate}
\end{defn}

\bigskip

The following facts about hyperbolic jigsaws are clear: 
	\begin{enumerate} 
		\item An $\SSS$-jigsaw $J$ is a fundamental polygon for $\Gamma_J$. The balancing condition ensures that $S_J$ is a {\it complete} hyperbolic orbifold with one cusp and $N+2$ cone points of order $2$, where $N=|J|$. Hence $\QQQ_J$ tessellates all  of $\HH$, and  $\TTT_J$ triangulates all of $\HH$.

\item If $\SSS=\{\Delta(k_1,k_2,k_3)\}$ and $k_1,k_2,k_3$ are all distinct,  then a $\SSS$-jigsaw group is a finite index subgroup of the {\em Weierstrass group} $\Gamma(k_1,k_2,k_3)$ and hence all such jigsaw groups belong to the commensurability class of $\Gamma(k_1,k_2,k_3)$. The question of which such groups are pseudomodular for the parameters $k_1,k_2,k_3 \in \QQ$ with $k_1k_2k_3=1$ is essentially Question 1 of \S 6 of \cite{LR}.
\item If $\SSS=\{\Delta(k_1,k_2,k_3)\}$ and exactly two of $k_1,k_2,k_3$ are the same, then the jigsaw group may not be a finite index subgroup of $\Gamma(k_1,k_2,k_3)$ as there is some choice involved in how the triangles are glued.

\item
The $\pi$ rotation about the marked point of an exterior side $e \in \mathcal E_J$ is an element of $\Gamma_J$. The $\pi$-rotation about the marked point of an interior side $e \in \mathcal F_J \setminus \mathcal E_J$ is not in $\Gamma_J$.

\item The set of vertices of $\QQQ_J$ coincide with the set of vertices of $\TTT_J$. This set is precisely the cusp set of $\Gamma_J$ since $S_J$ has one cusp.

\end{enumerate}

\bigskip

We have the following results, either of  which immediately implies   Theorem \ref{thm:Main}:

\begin{thm}\label{thm:S12}
	Let $\mathcal S=\{\Delta(1,1,1), \Delta(1,1/2,2)\}$. Then for every $\mathcal S$-jigsaw $J$, $\Gamma_J$ is a pseudomodular group. Furthermore, these groups are in infinitely many commensurability classes.
\end{thm}

	\begin{thm}\label{thm:S13}
		Let $\mathcal S=\{\Delta(1,1,1), \Delta(1,1/3,3)\}$. 
		Let $J_{A}$ be the jigsaw consisting of four tiles: one $\Delta(1,1,1)$ tile in the middle glued to three $\Delta(1,1/3,3)$ tiles along its three sides. Let $J$ be an $\SSS$-jigsaw with corresponding group $\Gamma_J$ and surface $S_J$. Then
		\begin{itemize}
			\item $\Gamma_J$ is arithmetic if and only if $J$ can be assembled from  a finite number of copies of  $J_{A}$ glued along matching sides. In particular, $\Gamma_{J_A}$ is arithmetic and all other arithmetic $\SSS(1,3)$ jigsaw groups are subgroups of $\Gamma_{J_A}$. 
			\item  $\Gamma_J$ is pseudomodular  if and only if the surface $S_J$ does not contain a closed geodesic $\gamma$ with $W$-cutting sequence $\overline{3,1,1/3,1}$ and it is not of the type above.
			\item $\Gamma_J$ is neither arithmetic nor pseudomodular if and only if the surface $S_J$  contains a closed geodesic $\gamma$ with $W$-cutting sequence $\overline{3,1,1/3,1}$. In this case, $\Gamma_J$ contains hyperbolic elements with integer fixed points (specials).
			\item There are infinitely many commensurability classes of  pseudomodular $\Gamma_J$, and infinitely many commensurability classes of non-pseudomodular $\Gamma_J$.
		\end{itemize}

		\end{thm}
		
		Note that  a geodesic $\gamma$ in Theorem \ref{thm:S13}  with $W$-cutting sequence $\overline{3,1,1/3,1}$  necessarily only intersects $\Delta(1,1/3,3)$ tiles, so the $V$ cutting sequence is $\overline{\Delta(1,1/3,3)}$. In particular,  any $\SSS$-jigsaw $J$ such that every $\Delta(1,1/3,3)$ tile of $J$ is matched to a $\Delta(1,1,1)$ tile (along the side with parameter $1$)  is arithmetic or pseudomodular.
		
		\medskip

We will prove Theorems \ref{thm:S12} and \ref{thm:S13} by analysing the relevant cusp sets, checking  the arithmeticity criteria and using Margulis's criteria to show that there are  infinitely many commensurability classes of pseudomodular groups coming from the constructions.

\section{Killer intervals }\label{s:killer}
We give a brief description of Long and Reid's method to show when a group has cusp set all of $\QQ \cup \{\infty\}$.

Let $\Gamma < \PGLtwoQ$ be a zonal Fuchsian group (that is, $\infty$ is fixed by a parabolic subgroup of $\Gamma$) such that $\HH/\Gamma$ is a complete, finite area hyperbolic surface with exactly one cusp. We are interested in the cusp set of $\Gamma$, which in this case (since the surface has only one cusp) is $$K_{\Gamma}:=\{g(\infty) ~|~ g \in \Gamma\}.$$
 Since $\Gamma< \PGLtwoQ$, $K_{\Gamma} \subset \QQ\cup \{\infty\}$. The question is in what situation is $K_{\Gamma}$ exactly equal to  $\QQ\cup \{\infty\}$.
For this, Long and Reid introduced the concept of a killer interval, which generalizes the idea of ``flipping'' a rational number $p/q \in (-1,1)$ by $ \left(
\begin{array}{cc}
0 & 1 \\
-1 & 0 \\
\end{array} \right) \in \PSLtwoZ$, to obtain a new rational number with strictly smaller (absolute) denominator. We have:

\begin{prop}\label{p:interval}(cf. \cite{LR} examples 1 and 2)
	Suppose \[g = \left(
	\begin{array}{cc}
	\alpha & \beta \\
	\gamma & \delta \\
	\end{array}
	\right) \in \Gamma <\PGLtwoQ, \quad \hbox{where}\qquad \alpha, \beta, \gamma, \delta \in \ZZ, \quad  \gcd(\alpha, \beta, \gamma, \delta)=1.\]
	Suppose further that $\gcd(\alpha, \gamma)=k$ so $\alpha=k\alpha'$, $\gamma=k\gamma'$ where $\gcd (\alpha', \gamma')=1$. Then for any $p/q \in (\alpha'/\gamma' -1/\gamma,\alpha'/\gamma' +1/\gamma)$, $g^{-1}(p/q)$ has strictly smaller denominator than $p/q$.
\end{prop}

\begin{proof}
	We have,
\[g = \left(
\begin{array}{cc}
\alpha & \beta \\
\gamma & \delta \\
\end{array}
\right), \qquad \alpha, \beta, \gamma, \delta \in \ZZ, \qquad  \hbox{gcd}(\alpha, \beta, \gamma, \delta)=1.\]
Then $\alpha/\gamma=\alpha'/\gamma'=g(\infty)$, where  $\alpha=k\alpha'$, $\gamma=k\gamma'$ and $k=\hbox{gcd}(\alpha, \gamma)$. For any rational $p/q \in \QQ$ where $\gcd(p,q)=1$,  \[g^{-1}(p/q)=\frac{\delta p -\beta q}{-\gamma p+\alpha q}= \frac{\delta p -\beta q}{k(-\gamma' p+\alpha' q)}. \]
Hence, the image has smaller denominator if \[|-\gamma p+\alpha q|<|q|, \quad \hbox{ that is,} \quad 
p/q \in (\alpha'/\gamma' -1/k\gamma',\alpha'/\gamma' +1/k\gamma'). \] 

\end{proof}

\begin{defn} (Killer intervals and contraction constant)
	The open interval $I= \left(\alpha'/\gamma'-1/k\gamma', \alpha'/\gamma'+1/k\gamma'\right)$ is called the {\em killer interval} about the cusp $\alpha'/\gamma'$ and $k\in \NN$ the {\em contraction constant} of the cusp $\alpha'/\gamma'$. 
	
\end{defn}

We also have a useful variation of Proposition \ref{p:interval} where  translation by an integer value does not affect the contraction constant:

\begin{prop}\label{p:intervalshift} 	
	Suppose \[g = \left(
	\begin{array}{cc}
	\alpha & \beta \\
	\gamma & \delta \\
	\end{array}
	\right) \in \PGLtwoQ, \quad \hbox{where}\qquad \alpha, \beta, \gamma, \delta \in \ZZ, \quad  \gcd(\alpha, \beta, \gamma, \delta)=1.\]
	Suppose further that $\gcd(\alpha, \gamma)=k$ and $\alpha=k\alpha'$, $\gamma=k\gamma'$.
 Let  $h=T^{n}gT^{-n}$ where $ n \in \ZZ$.
	
	Then for any $p/q \in (\alpha'/\gamma'+n -1/k\gamma',\alpha'/\gamma' +n+1/k\gamma')$, $h^{-1}(p/q)$ has strictly smaller denominator than $p/q$.
\end{prop}
\begin{proof} The proof is a simple computation and will be omitted.
	\end{proof}
\bigskip

Now suppose that $L$ is the smallest positive integer such that  $T^L \in \Gamma$. Then every $x \in \RR$ can be moved into a fundamental interval of length $L$, say $[0,L]$, by a suitable power of $T^L$ without increasing its denominator. If we can cover the fundamental interval $[0,L]$ by a finite set of killer intervals, then every rational number can be taken  to $\infty$ by a finite composition of elements of $\Gamma$, so that $K_{\Gamma}=\QQ \cup \{\infty\}$. We can summarize the above discussion as:

\begin{prop}(Long and Reid \cite{LR} Theorem 2.5)\label{prop:killer}
	 Let  $\Gamma < \PGLtwoQ$ be such that $\HH/\Gamma$ is a complete hyperbolic surface with one cusp  and suppose that $L$ is the smallest positive integer such that  $\left(
	 \begin{array}{cc}
	 1 & L \\
	 0 & 1 \\
	 \end{array}
	 \right) \in \Gamma$. If the fundamental interval $[0,L]$  can be covered by a finite set of killer intervals, then the cusp set $K_{\Gamma}=\QQ \cup \{\infty\}$.
\end{prop} 

Our basic strategy will be to construct Jigsaw groups such that Proposition \ref{prop:killer} applies.

\section{Basic properties of hyperbolic jigsaw groups}\label{s:basic}

In this section we explore some of the basic properties of  hyperbolic jigsaws and their associated groups and surfaces. We will mostly be interested in jigsaws composed from tiles which have some kind of integrality condition. We first establish some notation and convention. 

\subsection{Integral hyperbolic jigsaws}\label{ss:integraljigsaws} 

For $n \in \NN$, let $$\Delta^{(n)}:=\Delta(1, 1/n, n)=\Delta(1/n,n,1)=\Delta(n,1,1/n).$$ 

\begin{defn}
	An integral hyperbolic jigsaw set is a set of the form
	 \[\SSS(n_1,n_2, \ldots, n_s):=\{\Delta^{(n_1)},\Delta^{(n_2)}, \dots, \Delta^{(n_s)} \}\] where $1=n_1 <n_2<\cdots <n_s$. An integral hyperbolic jigsaw is an $\SSS$-jigsaw where $\SSS$ is integral. 
	
\end{defn}
Note that we have set $\Delta^{(n_1)}=\Delta^{(1)}$. This is convenient for our purposes although not completely necessary, we could have defined integral jigsaw sets which do not contain $\Delta^{(1)}$.

 Henceforth  $\SSS(n_1,n_2, \ldots, n_s)$ will denote an integral jigsaw set with $1 =n_1<n_2 \cdots < n_s$.

\begin{defn} For an integral jigsaw $J$, we will call sides of $J$,  $\QQQ_J$ or $\TTT_J$  with label/parameter in the set $\{n_j,1/n_j\}$ {\em type $n_j$} sides. So for example, $\SSS(1,2)$ jigsaws  can have only type 1 or type 2 sides, and $\SSS(1,3)$ jigsaws can have only type 1 or type 3 sides.

\end{defn}
\medskip
 Let  $J$ be a $\SSS(n_1,n_2, \ldots, n_s)$-jigsaw, with  $|J|=N$ and $sgn(J)=(m_1, \ldots, m_s)$. We will always normalize by putting one of the $\Delta^{(1)}$ tiles of $J$ into standard position $\Delta_{[\infty,-1,0]}$. Let $v_0, v_1, \ldots, v_{N+1}$ be the (cyclically ordered) vertices of $J$ where $v_0=\infty$, $v_1<v_2<\cdots <v_{N+1}$. Because of our normalization, $v_j=-1$, $v_k=0$ for some $1 \le j <k \le N+1$. Let $s_i=[v_i,v_{i+1}]$ ($v_{N+2}=v_0$) be the $i$th (exterior) side of $J$ (in cyclic order) with marked point $x_i \in s_i$. Let $\imath_i$ be the $\pi$-rotation about $x_i$, $i=0, \dots, N+1$. Then $\Gamma_J=\langle \imath_0, \imath_1,\ldots, \imath_{N+1}\rangle $, and $J$ is a fundamental domain for $\Gamma_J$. 	Recall that $\QQQ_J$ and $\TTT_J$ are the tessellation and triangulation of $\HH$ associated to $J$. We have: 
\medskip

\begin{prop}\label{p:basic} Let $J$ be an integral   $\SSS(n_1,n_2, \ldots, n_s)$-jigsaw with  $sgn(J)=(m_1, \ldots, m_s)$ and  $|J|=N$, and such that one of the $\Delta^{(1)}$ tile of $J$ is in standard position.  Let $\Delta_{[v_1, v_2,v_3]}\subset \TTT_J$ with sides $e_1=[v_1,v_2], e_2=[v_2,v_3]$ and $e_3=[v_3,v_1]$ labeled by $k_1,k_2$ and $k_3$  respectively. Suppose that $v_1=\infty$. Then:
	
	\begin{enumerate}

		\item if $\Delta_{[\infty, v_2,v_3]}$ is a $\Delta^{(1)}$ triangle, then $v_2=m$, $v_3=m+1$ for some integer $m$.
		
		\item if $\Delta_{[\infty, v_2,v_3]}$ is a $\Delta^{(n)}$ triangle where $n>1$ and $k_1=1$ (so $k_2=1/n$, $k_3=n$), then $v_2=m$, $v_3=m+1$ for some integer $m$. The triangle to the  right of $\Delta$ in $\TTT_J$  is a $\Delta^{(n)}$ triangle.
		
		\item if $\Delta_{[\infty, v_2,v_3]}$ is a $\Delta^{(n)}$ triangle where $n>1$ and $k_1=n$, (so $k_2=1$, $k_3=1/n$) then $v_2=m$, $v_3=m+n$ for some integer $m$. The triangles to the left and right of $\Delta$ in $\TTT_J$  are also $\Delta^{(n)}$ triangles.

		\item if $\Delta_{[\infty, v_2,v_3]}$ is a $\Delta^{(n)}$ triangle where $n>1$ and $k_1=1/n$ (so $k_2=n$, $k_3=1$), then $v_2=m$, $v_3=m+1$ for some integer $m$. The triangle to the  left of $\Delta$ in $\TTT_J$  is a $\Delta^{(n)}$ triangle.
		
		\item if $e=[\infty, m]$ is a side of $\TTT_J$ of type $n$, where $n \in \{n_1, \ldots, n_s\}$, then the marked point on $e$ is $m+\sqrt{n}i$, that is, it has height $\sqrt{n}$ above the real line.

			\end{enumerate}

		\end{prop}

\begin{proof}  By our normalization, the triangle $\Delta_{[\infty, -1,0]}$  is a $\Delta^{(1)}$ triangle of $\TTT_J$, so satisfies condition (1). The proposition follows easily by induction moving to the right and left of this normalized triangle, using the matching conditions.

\end{proof}
	It follows from the above that if $\Delta_{[\infty, v_2,v_3]}\subset \TTT_J$, then $v_3-v_2$ takes values in $\{1, n_2, \ldots, n_s\}$ and that the $\Delta^{(n_i)}$ triangles where $n_i>1$ with $\infty$ as a vertex occur in successive triples, where the middle triangle has width $n_i$ and the other two  width $1$.
 We can use this to define a $J$-width at each vertex of a triangle, and hence a $J$-width for each vertex of the jigsaw $J$:

\begin{defn} (Vertex $J$-widths)
	For any triangle $\Delta^{(n)}$, $n \in \NN$, the $J$-width of the vertices between the type 1 side and the type $n$ sides  is $1$. The $J$-width of the vertex between the two type $n$ sides is $n$. For the jigsaw $J$, the $J$-width of a vertex $v$ of $J$, denoted by $\JW(v)$ is the sum of the $J$-widths of the vertices of the triangles at $v$. 
\end{defn}

\noindent Remark. The $J$-width of a vertex is in general different from  the weight of the vertex $v$, which is just the number of triangles at  $v$. We denote the weight by $\wt(v)$.

\medskip

From this we easily get:

\begin{prop}
	Let $J$ be a $\SSS(n_1, \ldots, n_s)$-jigsaw as in Proposition \ref{p:basic}, with size $N$ and signature $(m_1, m_2, \ldots, m_2)$,  and let $\Gamma_J=\langle \imath_0, \imath_1,\ldots, \imath_{N+1}\rangle $. Then
	\[\imath_{N+1}\imath_N\cdots\imath_{0}=\left(
			\begin{array}{cc}
			1 & L \\
			0 & 1 \\
			\end{array}
			\right) \] where \[L=\sum_{i=1}^{s}m_i(2+n_i)\in \ZZ.\] 
	
\end{prop}

\begin{proof}
	We only need to compute the total $J$-width about the cusp of $\Gamma_J$. Each $\Delta^{(n)}$ triangle contributes a $J$-width of $2+n$, so the result follows.
	
	\end{proof}

	To summarize, we have shown:
 \begin{itemize}
	\item all (normalized) integral hyperbolic jigsaw groups have a fundamental interval of integral length $L=\sum_{i=1}^{s}m_i(2+n_i)$;
	\item  if $\Delta_{[\infty, v_2,v_3]} \subset \TTT_J$, then $v_2 \in \ZZ$ and $v_3-v_2$ takes values in $\{1,n_2,\ldots, n_s\}$;
	\item if $\Delta_{[\infty,v_2,v_3]}\subset \TTT_J$ has width $n=v_3-v_2>1$, then $\Delta_{[\infty,v_2,v_3]}$ is the lift of a $\Delta^{(n)}$ tile, and the triangles to the left and right are also lifts of $\Delta^{(n)}$ tiles with width  one. 
\end{itemize}    

More importantly, the cusps at the end of vertical sides of $\TTT_J$ have large killer intervals:

\begin{prop}\label{p:integercusps}
	 Let $J$ be an integral hyperbolic jigsaw satisfying the conditions of Proposition \ref{p:basic}. If $\Delta_{[\infty, q,r]}\subset \TTT_J$, then $q \in \ZZ$ is a cusp of $\Gamma_J$ with contraction constant one, so the killer interval about $q$ is $(q-1,q+1)$.   
\end{prop}

\noindent Note that it is possible that the side $e=[\infty,q] \in \mathcal F_J \setminus \mathcal E_J$, that is, $e$ may be an interior side.

\medskip
\begin{proof} This is basically a simple computation. We have $\Delta_{[\infty, q,r]}\subset \TTT_J$ is the lift of some tile of $J$, we consider another lift $\Delta_{[w_1,w_2,w_3]}$ of the same tile of $J$ such that $w_2=\infty$ (this is always possible since $\Gamma_J$ has only one cusp). Then there is an element $g \in \Gamma_J$ which takes the ordered triple $(w_1,w_2, w_3)=(w_1,\infty, w_2)$ to $(v_1, v_2,v_3)=(\infty, q,r)$, and furthermore, $g$ is completely determined by these pair of  triples. We separate the analysis into the three cases depending on the value of the parameter $k$ at the side $[\infty, q]$.
	
\begin{itemize}
	\item Case 1: $k=1$. In this case, $r=q+1$ and $w_3=m$, $w_1=m+1$ for some integer $m$. Then,
	
	\[
	g= 
	\left( \begin{array}{cc} 
	q & -1-q-mq \\  
	1 & -1-m \\
	\end{array} \right)
	\]
	and hence the killer interval about $q$ is $(q-1,q+1)$ by proposition \ref{p:interval}.
	\item Case 2: $k=n\in \{n_2, \ldots, n_s\}$. In this case, $r=q+n$ and $w_3=m$, $w_1=m+1$ for some integer $m$. Then,
	\[g= 
	\frac{1}{\sqrt{n}} \left( \begin{array}{cc} 
	q & -n-q-mq \\  
	1 & -1-m \\
	\end{array} \right)
	\]
	and hence the killer interval about $q$ is $(q-1,q+1)$ by proposition \ref{p:interval}.
	\item Case 3:   $k=1/n$ where $n\in \{n_2, \ldots, n_s\}$. In this case, $r=q+1$ and $w_3=m$, $w_1=m+n$ for some integer $m$. Then,
	\[g= 
	\frac{1}{\sqrt{n}} \left( \begin{array}{cc} 
	q & -n-nq-mq \\  
	1 & -n-m \\
	\end{array} \right)
	\]
	and hence the killer interval about $q$ is $(q-1,q+1)$ by proposition \ref{p:interval}.
\end{itemize}

\end{proof}


We note that any  other lift $\Delta_{[w_1,w_2,w_3]}$ of the same tile of $J$ such that $w_2=\infty$ is a translation of the one we choose by some multiple of the fundamental interval $L$. Hence, if $q$ is a cusp and $g\in \Gamma_J$ maps $\infty$ to $q$, then the set of  $A \in \Gamma_J$ mapping $\infty$ to $q$ is  \[\{A \in \PGLtwoQ ~|~ A=gT^{LM}, \quad M \in \ZZ\}.\]

\medskip

We next analyse cusps at the next level away from $\infty$:

\begin{prop}\label{p:nextcusps}
	Let $J$ be a hyperbolic jigsaw satisfying the conditions of proposition \ref{p:basic}. Suppose  that $\Delta_{[\infty, m, m+n]} \subset \TTT_J$ where $n>1$ is the lift of a $\Delta^{(n)}$ tile. Then $\Delta_{[m+n,m,m+n/2]}\subset \TTT_J$. 
	The contraction constant of the cusp $m+n/2$ is one if $n$ is odd, and either one or two, if $n$ is even.
\end{prop}


\begin{proof} We first note that $\Delta_{[v_1,v_2,v_3]}:=\Delta_{[m+n,m,m+n/2]}\subset \TTT_J$ since the parameter at the side $[m, m+n]$ is one. Then the parameters at the sides $[m, m+n/2]$ and $[m+n/2,m+n]$ are $1/k$ and $k$ respectively, where $k \in \{n_1, \ldots, n_s\}$. Let $\Delta_{[w_1,w_2,w_3]}$ be another lift of the same tile as $\Delta_{[v_1,v_2,v_3]}$ where $w_3=\infty$ corresponds to $v_3=m+n/2$. Then the parameter of the side $[w_3,w_1]$ is $k$, so $w_1=r$, $w_2=r+k$ for some $r \in \ZZ$. Hence there is an element $g \in \Gamma_J$ taking the ordered triple $(\infty, r, r+k)$ to $(m+n/2, m+n, m)$. A direct computation gives
	\[g= 
	\frac{1}{\sqrt{kn}} \left( \begin{array}{cc} 
	2m+n & -r(2m+n)-k(m+n)\\  
	2 & -2r-k \\
	\end{array} \right)
	\] 
	Hence the contraction constant at $m+n/2$ is one if $n$ is odd, and either one or two if $n$ is even, depending on whether $k$ is odd or even respectively.

\end{proof}

We next show that many of the sides of an integral $J$-jigsaw cannot be type $1$, which will be useful when we are determining if the group $\Gamma_J$ is arithmetic or not.

\begin{prop}\label{p:sidelabel}
	Let $J$ be a $\SSS(n_1,n_2, \ldots, n_s)$-jigsaw. For any $n \in \{n_2, \ldots n_s\}$, at least two of the  sides of $J$ are type $n$.
\end{prop}

\begin{proof}
	Let $n \in \{n_2, \ldots n_s\}$. By definition, $J$ contains a $\Delta^{(n)}$ tile. If its two type $n$ sides  are already sides of $J$, we are done. Otherwise, say one of the type $n$ sides is not a side of $J$. It must be matched to another $\Delta^{(n)}$ tile which now has a new type $n$ side. This is either a side of $J$ or matched to another $\Delta^{(n)}$ tile. Proceeding inductively, since $|J|$ is finite, we eventually must have at least two sides of $J$ which are type $n$. 
	\end{proof}

\subsection{Non-integral jigsaws}\label{ss:non-integral} Much of the discussion above can be applied to non-integral hyperbolic jigsaws but without the integrality conditions, it is difficult to obtain useful conclusions. For example, the fundamental interval $[0,L]$ may not have integral length, the vertical sides of $\TTT_J$ may not have integer ends, and typically, the contraction constant about a cusp at the end of a vertical side of $\TTT_J$ may be greater than one. The analysis of such jigsaws is therefore more complicated.

\section{Cusp set of $\SSS(1,2)$ Jigsaw groups}\label{s:onetwojigsaws} We apply the results of \S\ref{ss:integraljigsaws} to $\SSS(1,2)$ jigsaws which are particularly simple.  
What is interesting is that when we combine  $\Delta(1,1,1)$ tiles with $\Delta(1,1/2,2)$ tiles, the arithmetic property of the Weierstrass groups $\Gamma(1,1,1)$ and $\Gamma(1,1/2,2)$ is destroyed while the property that the cusp set is $\QQ \cup \{\infty\}$ is preserved. In this section, we analyse the cusp set.

\subsection{The Weierstrass groups  $\Gamma(1,1,1)$ and $\Gamma(1,1/2,2)$.} We first note that the Weierstrass group $\Gamma(1,1,1)$ associated to the tile $\Delta(1,1,1)$ is just an index three subgroup of the modular group $\PSLtwoZ$, hence arithmetic. In standard position, the three generators are 
\[\imath_1 = \left(
\begin{array}{cc}
1 & 2 \\
-1 & -1 \\
\end{array}
\right), \quad
\imath_2 = \left(
\begin{array}{cc}
1 & 1 \\
-2 & -1 \\
\end{array}
\right), \quad
\imath_3 = \left(
\begin{array}{cc}
0 & 1 \\
-1 & 0 \\
\end{array}
\right).\] 
The Weierstrass group $\Gamma(1,1/2,2)$ is also arithmetic. In standard position, the three generators  are
\[
\imath_1 = \left(
\begin{array}{cc}
	1 & 2 \\
	-1 & -1 \\
\end{array}
\right), \quad
\imath_2 =\frac{1}{\sqrt{2}} \left(
\begin{array}{cc}
	2 & 2 \\
	-3 & -2 \\
\end{array}
\right), \quad
\imath_3 = \frac{1}{\sqrt{2}}\left(
\begin{array}{cc}
	0 & 2 \\
	-1 & 0 \\
\end{array}
\right).
\]  
It is commensurable to the group associated to the square torus and also the Hecke group $G_4$. In particular  $(\tr\imath_1\imath_2)^2$, $(\tr\imath_2\imath_3)^2$ and $(\tr\imath_3\imath_1)^2$ are all integers. Note that $x_1$, the fixed point of $\imath_1$, is at height one above the real axis and $x_3$, the fixed point of $\imath_3$, is at height $\sqrt{2}$ above the real axis.

\subsection{Killer intervals and cusp set of $\SSS(1,2)$-jigsaws} We have:
\begin{prop}\label{cuspset12}
	If $J$ is an $\SSS(1,2)$ jigsaw, then $K_J$, the cusp set of $\Gamma_J$   is $\QQ \cup \{\infty\}$.
\end{prop} 

\begin{proof}
	This follows easily from proposition \ref{p:basic}, \ref{p:integercusps} and \ref{p:nextcusps} as applied to $\SSS(1,2)$ jigsaws. First,  if $\Delta:=\Delta_{[\infty,v_2,v_3]} \subset \TTT_J$  and the parameter at $s_1=[\infty, v_2]$ is $k_1=2$, then by proposition \ref{p:basic}, \\
	(i) $v_2=m$, $v_3=m+2$ for some integer $m$, and \\
	(ii) $k_2=1$, $k_3=1/2$. \\
	(iii) the triangles of $\TTT_J$ to the left and right of $\Delta$ have width one.
	
	By proposition \ref{p:nextcusps}, $m+1$ is a vertex of the  triangle $\Delta'$ of $\TTT_J$ which shares the side $[m,m+2]$ with $\Delta$ (since $k_2=1$). All other types of triangles of $\TTT_J$ with $\infty$ as a vertex have width one. Hence, every $m \in \ZZ$ is a cusp of $\Gamma_J$.  Next, by proposition \ref{p:integercusps}, if $[\infty, m]$ is a side of $\TTT_J$, then $m \in \ZZ$ and the killer interval about $m$ is $(m-1,m+1)$. It now follows 
	that the set of killer intervals about the integers cover all of $\RR$ and Proposition \ref{prop:killer} applies.
	\end{proof}

\section{Cusp set of $\SSS(1,3)$ Jigsaws}\label{s:onethreejigsaws}
In this section, we consider $\SSS(1,3)$-jigsaws and show that the cusp set is either all of $\QQ \cup \{\infty \}$, or there are specials (fixed points of hyperbolic elements of $\Gamma_J$) which are in the equivalence classes of certain integers. The latter case occurs if and only if there is a closed geodesic $\gamma \subset S_J$ with $W$-cutting sequence $\overline{1/3,1,3,1}$.

\subsection{The Weierstrass group $\Gamma(1,1/3,3)$.}
In standard position, we get from (\ref{eqn:involutions}) that $\Gamma(1,1/3,3)$ is generated by
 \begin{equation}
 \imath_1=\left(
 \begin{array}{cc}
1 & 2 \\
 -1 & -1 \\
 \end{array}
 \right), \quad
 \imath_2=\frac{1}{\sqrt{3}} \left(
 \begin{array}{cc}
 3 & 3 \\
 -4 & -3 \\
 \end{array}
 \right), \quad
 \imath_3=\frac{1}{\sqrt{3}} \left(
 \begin{array}{cc}
 0 & 3 \\
 -1 & 0 \\
 \end{array}
 \right)\end{equation}

\begin{figure}[hbt]\centering{\includegraphics[height=6cm]{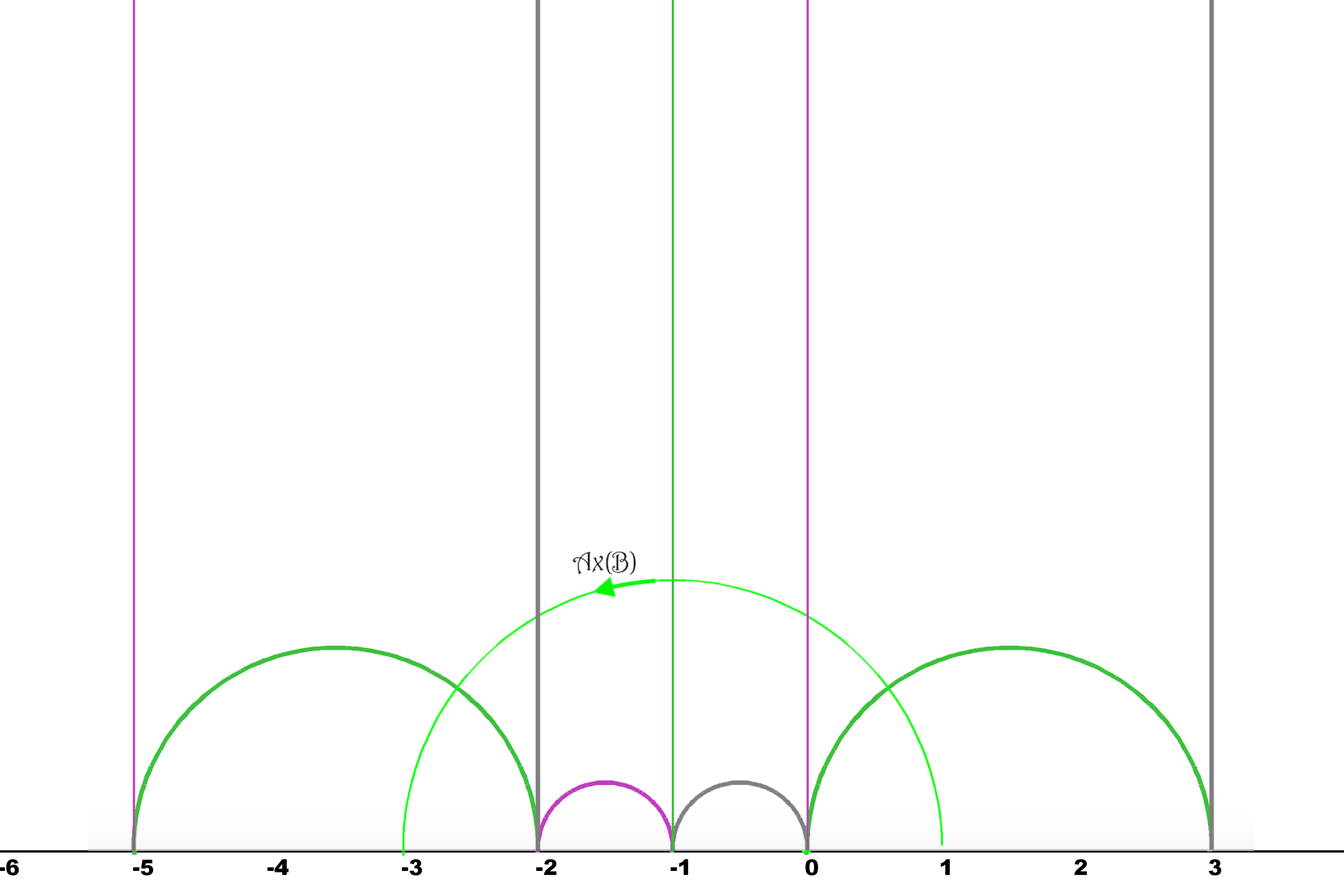}}
	\caption{Tessellation associated to $\Gamma(1,1/3,3)$ and axis of the hyperbolic element $B$ with endpoints at $-3$ and $1$.}  
	\label{fig:Deltathree} 
\end{figure}

 A fundamental interval for $\Gamma(1,1/3,3)$ is of length $5$, we may take this to be the closed interval $[-2,3]$, with vertical sides $[\infty, -2]$, $[\infty, -1]$, $[\infty, 0]$ and  $[\infty, 3]$  within this interval, see figure \ref{fig:Deltathree}. The killer intervals about $-2, -1,0$ and  $3$  all have width $2$, they are $(-3,-1)$, $(-2,0)$, $(-1,1)$ and $(2,4)$  respectively. By proposition \ref{p:nextcusps}, the triangle $\Delta_{[0,3/2,3]}$ adjacent to $\Delta_{[\infty,0,3]}$  is a triangle of $\TTT$ and the contraction constant at $3/2$ is $1$, so the killer interval for the cusp $3/2$ as $(1,2)$. It follows that the killer intervals about $-2, -1,0,3/2$ and $3$ cover all of $[-2,3]$ except possibly the points $1$ and $2$. In fact, these are specials and cannot be covered by any killer intervals. A direct computation shows that   $x_1=-3$ and $x_2=1$ are the attracting and repelling fixed points of the element \[B=(\imath_1 \imath_2\imath_1)\imath_3=\frac{1}{3}\left(
 \begin{array}{cc}
 7 & -6 \\
 -2 & 3 \\
 \end{array}
 \right)
 \in \Gamma(1,1/3,3)\] and $y_1=2$, $y_2=6$ are the attracting and repelling fixed points of 
 \[\imath_3\imath_2\imath_1B(\imath_3\imath_2\imath_1)^{-1}=\frac{1}{3}\left(
 \begin{array}{cc}
 -3 & 24 \\
 -2 & 13 \\
 \end{array}
 \right) \in \Gamma(1,1/3,3).\]  Since $B$ is the product of the two involutions $\imath_1\imath_2\imath_1$ and $\imath_3$, its invariant axis passes through the fixed points $-2+\sqrt{3}i$ and $\sqrt{3}i$ of these elements ( in fact, it also passes through  $-1+2i$), and $B$ translates by twice the distance between the points $-2+\sqrt{3}i$ and $\sqrt{3}i$ along this axis. The $W$-cutting sequence of the axis is $\overline{1/3,1,3,1}$. Furthermore, $i_3$ interchanges $1$ and $-3$. It follows that every rational number is either a cusp or is equivalent to the special $1$. If we consider the punctured torus $T$ which double covers the surface $S=\HH/\Gamma(1,1/3,3)$ (with the induced labeled triangulation), there is a (unique) simple closed geodesic $\gamma\subset T$ which has $W$-cutting sequence  $\overline{1/3,1,3,1}$ whose fixed points are equivalent to $1$ or $-3$.


 \subsection{Killer intervals and cusp set of $\SSS(1,3)$-jigsaws}
 \begin{prop}\label{p:cuspset13}
 	Let $J$ be a $\SSS(1,3)$ jigsaw. The cusp set $K_{\Gamma}$ of $\Gamma_J$ is a proper subset of  $\QQ \cup \{\infty\}$ if and only if the surface $S_J$ contains a closed geodesic $\gamma$ with $W$-cutting sequence $\overline{1/3,1,3,1}$.
 \end{prop}

\begin{proof} ($\Longleftarrow$)
 Suppose that $S_J$ contains a geodesic $\gamma$ with $W$-cutting sequence $\overline{1/3,1,3,1}$. Note that $\gamma$ only intersects $\Delta^{(3)}$ triangles, that is the $V$-cutting sequence of $\gamma$ is $\overline{\Delta^{(3)}}$. We may lift one of these triangles to $\Delta_{[\infty, m, m+3]}\subset \TTT_J$, where $m \in \ZZ$. Without loss of generality, suppose that $\widetilde{\gamma}$ intersects the sides $[\infty, m]$ and $[m, m+3]$ of $\Delta_{[\infty, m, m+3]}$ (which gives the $3,1$ part of the $W$-cutting sequence). From the analysis of the Weierstrass group $\Gamma(1,1/3,3)$, since $\widetilde{\gamma}$ has $W$-cutting sequence $\overline{1/3,1,3,1}$, we see that the endpoints of $\widetilde{\gamma}$ are $m-3$ and $m+1$ which are {\em specials} of $\Gamma_J$ (rational fixed points of hyperbolic elements of $\Gamma$). Hence, $K_{\Gamma}$ is a proper subset of $\QQ \cup \{\infty\}$. 
	
	\medskip
	
	\noindent ($\Longrightarrow$) Conversely, suppose that  $K_{\Gamma}$ is a proper subset of $\QQ \cup \{\infty\}$. By proposition \ref{p:basic}, all triangles $\Delta_{[\infty, q,r]} \subset \TTT_J$ have width either 1 or 3, and furthermore, the triangles with width 3 have triangles of width 1 on either side. The killer intervals about the cusps $q$ are $(q-1,q+1)$ by proposition  \ref{p:integercusps}. Furthermore, if $\Delta_{[\infty, q,q+3]} \subset \TTT_J$ has width three, then by proposition \ref{p:nextcusps}, $q+3/2$ is a cusp with killer interval $(q+1,q+2)$. It follows that the fundamental interval $[0,L]$ can be covered by the killer intervals except possibly for a finite number of points. These points are of the form $m+1, m+2$ where $\Delta_{[\infty, m, m+3]} \subset \TTT_J$. Since 
	$K_{\Gamma}$ is a proper subset of $\QQ \cup \{\infty\}$, there is at least one such point which is not covered by any killer interval. Without loss of generality, by conjugating by a suitable translation, and to simplify the computations, we may assume that $m=0$, and the point $1$ is not covered by any killer interval where $\Delta_{[\infty, 0, 3]}\subset \TTT_J$. In particular, $1$ is not a cusp of $\Gamma_J$. Let $\alpha$ be the geodesic ray from $y=-1+2i$ ending at $1$. Note that $\alpha$ is contained in the axis of 
	\[B=(\imath_1 \imath_2\imath_1)\imath_3=\frac{1}{3}\left(
	\begin{array}{cc}
	7 & -6 \\
	-2 & 3 \\
	\end{array}
	\right)
	\in \Gamma(1,1/3,3)\] 
	 The first two terms of the $V$-cutting sequence of $\alpha$  are $\Delta^{(3)}, \Delta^{(3)},\dots$, and the first three terms of the $W$-cutting sequence of $\alpha$ is $1,3,1, \dots$ where both cutting sequences are infinite as $1$ is not a cusp. We claim the $V$-cutting sequence of $\alpha$ cannot contain any $\Delta^{(1)}$, if it does, it will be finite which implies $1$ is a cusp, giving a contradiction. First, suppose that $\Delta_{[3,0,3/2]} \subset \TTT_J$ is a $\Delta^{(1)}$ tile. Then a direct computation gives that the neighboring triangle sharing the side $[0,3/2]$ with  $\Delta_{[3,0,3/2]}$ is $\Delta_{[3/2,0,1]}$ so that $1$ is a cusp. So the $V$-cutting sequence of $\alpha$ starts with at least four $\Delta^{(3)}$ triangles ( the fourth triangle is also  a $\Delta^{(3)}$ triangle since the side adjoining the third triangle has label in the set $\{3,1/3\}$). By a similar argument, if the $V$ cutting sequence of $\alpha$ contains a $\Delta^{(1)}$ triangle, the number of $\Delta^{(3)}$ triangles before the first $\Delta^{(1)}$ triangle in the cutting sequence is even , that is either $4k$ or $4k+2$ for some $k \in \NN$. Consider the action of $B^k$ on $\alpha$ and the triangles of $\TTT_J$ which it intersects. This maps the first $\Delta^{(1)}$ triangle intersecting $\alpha$ to either $\Delta_{[\infty, -1,0]}$ or $\Delta_{[3,0,3/2 ]}$, while fixing $1$. In either case, the next triangle intersected by $\alpha$ has $1$ as a vertex (either $\Delta_{[\infty, 0,1]}$ or $\Delta_{[3/2,0,1 ]}$). Acting  by $B^{-k}$ to move back to the original configuration, we conclude that if $\alpha$ intersects a $\Delta^{(1)}$ triangle, then the next triangle it intersects must have $1$ as a vertex. Hence, since $1$ is not a cusp, we conclude that the $V$-cutting sequence of $\alpha$ is $\overline{\Delta{(3)}}$ and from the analysis of $\Gamma(1,1/3,3)$ earlier,  the $W$ cutting sequence of $\alpha$ is $\overline{1,3,1,1/3}$. Projecting back to the surface $S_J$, we get a geodesic ray $\bar \alpha \subset S_J$ with cutting sequence $\overline{1,3,1,1/3}$. Since there are only a finite of sides in the triangulation of $J$, $\bar \alpha$ intersects the some side $e$ of the triangulation twice, say at $y_1, y_2 \in \bar{\alpha}$. Consider the closed curve on $S_J$ consisting of the geodesic segment from $y_1$ to $y_2$ along $\bar{\alpha}$  followed by the geodesic on $e$ joining $y_2$ back to $y_1$. Pulling this tight produces the closed geodesic $\gamma$ on $S_J$ with $W$-cutting sequence $\overline{1,3,1,1/3}$ as required. This completes the proof.

\end{proof}


\section{Criteria for (non)-arithmeticity}\label{s:non-arith}
We first establish some general results for integral jigsaw groups.
By results of Takeuchi \cite{Tak}, to show that a  non-cocompact Fuchsian group $\Gamma \le \PSLtwoR$ of finite co-area ( with no elements of order $2$) and with invariant trace field $\QQ$ is non-arithmetic, it suffices to find an element $\gamma \in \Gamma$ such that $\tr \gamma^2 \not\in \ZZ$. Conversely, to show that $\Gamma$ is arithmetic, by \cite{HLM}, it suffices to show that  $\tr \gamma^2 \in \ZZ$ for all $\gamma$ in a generating set of $\Gamma$.
Since $\tr \gamma^2=(\tr \gamma)^2-2$, we can check if $(\tr \gamma)^2 \in \ZZ$ instead.

Recall that for any integral jigsaw $J$ with $|J|=N$, $\Gamma_J$ is freely generated by $\imath_0, \imath_1, \ldots, \imath_{N+1}$, where $\imath_j$ is the $\pi$-rotation about the marked point $x_j$ of the side $s_j$ of $J$. So $\Gamma_J\cong \ZZ_2 \ast \cdots \ast \ZZ_2$, the  free product of $N+2$ copies of $\ZZ_2$.  We will apply  results of  \cite{Tak} and \cite{HLM} to the index two subgroup $\Gamma^{(2)}_J < \Gamma_J$ consisting of elements of $\Gamma_J$ with even word length. Then $\Gamma^{(2)}_J$ is a free group of rank $N+1$ and \[\Gamma^{(2)}_J=\langle \imath_0\imath_1, \imath_0\imath_2, \ldots, \imath_0\imath_{N+1} \rangle.\]
A fundamental domain for $\Gamma^{(2)}_J$ is $J \cup \imath_0(J)$ with opposite sides identified.

Let $J$ be a $\SSS(n_1, \ldots, n_s)$-jigsaw of size $|J|=N$, with a $\Delta^{(1)}$ tile in standard position. For  $i=0, \ldots, N+1$,  let $v_0=\infty, v_1, \ldots, v_{N+1}$ be the vertices of $J$  where $v_1<v_2 < \cdots <v_{N+1}$, $s_i=[v_{i}, v_{i+1}]$ be the sides of $J$ and $k_i \in \{1,n_2, 1/n_2,\ldots, n_s,1/n_s \}$ be the label (parameter) of the side $s_i$. For $j \in \ZZ$, let  $e_j:=[\infty, m_j]$ be successive (unoriented) vertical sides of $\QQQ_J$ indexed in the negative direction (so $m_j>m_{j+1}$ for all $j \in \ZZ$), normalized so $e_0=[\infty, m_0]=[\infty, v_1]=s_0$ (so $e_{-1}=[\infty, v_{N+1}]=s_{N+1}$).  By looking at the lifts of the fundamental domain $J$ in the developing map, from the definition of the $J$-widths, and by propositions \ref{p:basic} and \ref{p:sidelabel}, we easily get:

\begin{lem}\label{lem:verticalsides}
	 Let $J$ be a $\SSS(n_1, \ldots, n_s)$-jigsaw with $|J|=N$ and for $i=0, \ldots, N+1$, and $j \in \ZZ$, let $v_i, s_i$, $k_i$ and $e_j$ be defined as above. Then
	\begin{enumerate}
		\item For each $n \in \{n_2, \ldots, n_s\}$, at least two of the sides $s_k$, $s_l$ of $J$ are type $n$.
		\item If $j \equiv i \mod (N+2)$,  $e_j$ is a lift of the side $s_i$ in $\QQQ_J$ and so has label $k_i$;
		\item If $j \equiv i \mod (N+2)$, $m_{j-1}-m_j$,  is the $J$-width of $v_i$; 
		\item If $j \equiv i \mod (N+2)$,  and $k_i \in \{n,1/n\}$ where $n \in \{n_1, \ldots, n_2\}$, the marked point $y_j \in e_j$ is at height $\sqrt{n}$  above the real line and the $\pi$ rotation about $y_j$ is
		\[h_j:=\frac{1}{\sqrt{n}} \left(
		\begin{array}{cc}
		m_j & -(m_j^2+n) \\
		1 & -m_j \\
		\end{array}
		\right).\]
		Note that $h_j \in \Gamma_J$ has odd word length for all $j \in \ZZ$ (it is the conjugate of a generator) and hence  $h_kh_l \in \Gamma^{(2)}_J$ for all $k,l \in \ZZ$.
		
		\item $L:=[m_{N+2}, m_0]$ is a fundamental interval for $\Gamma_J$.
	\end{enumerate}
\end{lem}



\subsection{Non-arithmeticity of $\SSS(1,2)$ Jigsaws}

We first recall that if $T$ is a triangulation of a polygon $P$ with $n \ge 4$ sides, a triangle $\Delta \subset T$ is a {\em ear} if two of its sides are sides of $P$. We call the vertex $v$ between these two sides the {\em tip} of the ear, so $\wt(v)=1$. Every triangulation of $P$ has at least two ears. 

\begin{prop}\label{p:nonarith12}
	For any  $\SSS(1,2)$-jigsaw $J$, the group $\Gamma_J$ is non-arithmetic.
\end{prop}

\begin{proof}  Let $J$ be a $\SSS(1,2)$-jigsaw of size $|J|=N$  and for $i=0, \ldots, N+1$ and  $j \in \ZZ$, let $v_i, s_i$, $k_i$, $e_j$ and $h_j$ be defined as in Lemma \ref{lem:verticalsides}. By lemma \ref{lem:verticalsides}(1) and (2), some side $e_j$ is type 2. Hence, the rotation in the side $e_j$ is 
	\[
	h_j = 
	\frac{1}{\sqrt{2}} \left( \begin{array}{cc}
	m_j & -m_j^2-2\\
	1 & -m_j\\
	\end{array} \right)
	\]
	Suppose that $e_k$ is type 1 for $k \neq j$. A direct computation using Lemma \ref{lem:verticalsides} (4) shows that the rotation in the side $e_k$ is
	\[
	h_k = 
	\left( \begin{array}{cc}
	m_k & -m_k^2-1\\
	1 & -m_k\\
	\end{array} \right)
	\]
	Then $(\tr h_jh_k)^2 = \frac{1}{2}(-(m_j-m_k)^2 -3)^2 \in \ZZ$ if and only if $(m_j-m_k)^2 \equiv 1 \mod 2$ if and only if $m_j - m_k \equiv 1 \mod 2$. 
	\newline
	Similarly, if $e_k$ is type 2 for $k \neq j$, then the rotation in the side $e_k$ is
	\[
	h_k = 
	\frac{1}{\sqrt{2}} \left( \begin{array}{cc}
	m_k & -m_k^2-2\\
	1 & -m_k\\
	\end{array} \right)
	\]
	Then $(\tr h_jh_k)^2 = \frac{1}{4}(-(m_j - m_k)^2 -4)^2 \in \ZZ$ if and only if $(m_j-m_k)^2 \equiv 0 \mod 4$ if and only if $m_j - m_k \equiv 0 \mod 2$. 
	
	Hence, if $\Gamma_J$ is arithmetic,  the parity (mod 2) of the type 1 and type 2 vertical sides $e_j$  must be different. This immediately shows that many jigsaws are not arithmetic. For example,  if the jigsaw has a {\em ear} which is a $\Delta^{(1)}$ tile, then it cannot be arithmetic since the $J$-width between the two unmatched sides is 1 and they are both type 1. By Lemma \ref{lem:verticalsides}(3), this means there are two successive sides $e_k$, $e_{k+1}$ of type 1 and $m_k=m_{k+1}+1$. Then either $m_k$ or $m_{k+1}$ has the same parity as $m_j$ where $e_j$ is type 2, so either $(\tr h_jh_k)^2$ or $(\tr h_jh_{k+1})^2$  is not integral. 
	
	So for a jigsaw $J$ to be arithmetic, the $J$-width at each vertex $v_i$ (which lies between $s_{i-1}$ and $s_i$) has to be:
	\begin{itemize}
		\item even, if $s_{i-1}$ and $s_i$ are same types;
		\item odd, if $s_{i-1}$ and $s_i$ are different types.
	\end{itemize}
	
	We call a $\SSS(1,2)$-jigsaw $J$ which satisfies the two conditions above admissable.	Now, suppose for a contradiction that there exists an arithmetic $\Gamma_J$, so $J$ is admissible. 
	Let $\Delta$ be a ear of $J$. By the earlier discussion $\Delta$ is a $\Deltatwo$ triangle. It is easy to check that $J \setminus \Delta$ is again admissible. Proceeding inductively leads to a contradiction since $J$ has a finite number of $\Deltatwo$ tiles and has at least one $\Deltaone$ tile.

	
	\subsection{(Non)-arithmeticity of $\SSS(1,3)$ Jigsaws}
	
	\begin{prop}\label{p:nonarith13}
		For any  $\SSS(1,3)$-jigsaw $J$, if the group $\Gamma_J$ is arithmetic, then all the sides $s_i$ of $J$ must be type 3,  and the $J$-width at each vertex $v_i$ of $J$ is congruent to $0$ mod $3$.
	\end{prop}
	
	\begin{proof}
		This follows easily from Lemma \ref{lem:verticalsides} applied to $\SSS(1,3)$ jigsaws. Adopting the same notation, we first note that there exists a vertical side $e_j=[\infty, m_j]$ of $\QQQ_J$ of type 3. Then the $\pi$-rotation about the marked point on the side $e_j$ is 
		\[
		h_j = 
		\frac{1}{\sqrt{3}} \left( \begin{array}{cc}
		m_j & -m_j^2-3\\
		1 & -m_j\\
		\end{array} \right)
		\]		
		If some other vertical side $e_k=[\infty, m_k]$ is type 1, then a direct calculation shows that the $\pi$-rotation about the marked point on the side $e_k$ is 
		\[
		h_k = 
		\left( \begin{array}{cc}
		m_k & -m_k^2-1\\
		1 & -m_k\\
		\end{array} \right)
		\]	
		Then	 $(\tr h_jh_k)^2 = \frac{1}{3} (-(m_j - m_k)^2-4)^2 \not\in \ZZ$ (this depends on the fact that $3 \not\equiv 1 \mod 4$). 
		
		Similarly, if some other $e_k$ is type 3, then 
		\[
		h_k = 
		\frac{1}{\sqrt{3}} \left( \begin{array}{cc}
		m_k & -m_k^2-3\\
		1 & -m_k\\
		\end{array} \right)
		\]	
		Hence $(\tr h_jh_k)^2 = \frac{1}{9}(-(m_j-m_k)^2-6)^2 \in \ZZ$ if and only if $m_j-m_k \equiv 0 \mod 3$. 
		
		It follows that if $\Gamma_J$ is arithmetic, all sides $s_i$ of $J$ must be of type 3, and $\JW(v_i) \equiv 0 \mod 3$ for all vertices $v_i$ of $J$.
	\end{proof}
	
	It turns out that arithmetic $\SSS(1,3)$ jigsaws exist. We have:
	
	\begin{prop}\label{p:arithemtic13}
		The $\SSS(1,3)$ jigsaw $J_{A}$ with signature $(1,3)$ consisting of one $\Delta^{(1)}$ tile and  three $\Delta^{(3)}$ tiles, one attached to each of the three sides of the $\Delta^{(1)}$ tile, is arithmetic.
	\end{prop}
	
	\begin{proof}
		We can show by direct calculation that  $(\tr \imath_0\imath_i )^2\in \ZZ$, for $i=1,2,\ldots, 5$ where  $\Gamma_J^{(2)}=\langle \imath_0\imath_1, \ldots, \imath_0\imath_5 \rangle$. Hence, by \cite{HLM}, $\Gamma_J$ is arithmetic. We may also observe that $S_J$ has a six fold symmetry, under rotation by $\pi/3$ about the centroid of the the $\Delta^{(1)}$ tile, and that the quotient surface corresponds to the Hecke group $G_6$ which is known to be arithmetic.
		\end{proof}
	
	It follows that jigsaws composed of copies of $J_{A}$ glued along matching sides correspond to  finite covers of $S_{J_{A}}$, equivalently, to finite index subgroups of $\Gamma_{J_{A}}$, and are all arithmetic.
	In fact, the converse is also true. We will need the following:
	
	\begin{lem}\label{l:badvertices}
		Let $P$ be an ideal polygon with $n \ge 4$ sides and $T$ a triangulation of $P$. Then there exists at least two non-adjacent vertices $u$ and $v$ of $P$ with $\wt(u), \wt(v) \in \{2,3\}$  with respect to the triangulation.
	\end{lem}
	
	\begin{proof} 
		Without loss of generality, we may assume that $n \ge 5$, $n=4$ is trivial. Let $\Delta_1, \ldots, \Delta_a$ be the ears of $P$, and let $$P'=P \setminus\bigcup_{i=1}^a \Delta_i.$$
		If $P'$ consists of only one triangle, then $a=2$ or $3$ and a direct check shows the result holds. If $P'$ contains at least two triangles, then it has at least two ears (which are not ears of $P$) with tips $w_i$, $w_j$ which are not adjacent in $P'$. Then $w_i$ and $w_j$ are non-adjacent vertices of  $P=P' \cup \bigcup_{i=1}^a \Delta_i$ with $\wt(w_i), \wt(w_j) \in \{2,3\}$ as required.

		\end{proof}

	\begin{prop}\label{p:arith13all}
		A $\SSS(1,3)$ jigsaw $J$ is arithmetic if and only if it consists of a finite number of $J_{A}$ pieces glued along matching sides.
	\end{prop}
	
	\begin{proof}
		We will show that an arithmetic jigsaw $J$ can be decomposed into polygons of certain types which we call blocks. Call a polygon $B$ assembled from $\Deltaone$ and $\Deltathree$ tiles glued along matching sides a {\em block} if
		\begin{enumerate}
			\item  all the sides of $B$ are type 3; and 
			\item \begin{enumerate}
				\item  [(i)] $B$ consists of only $\Delta^{(3)}$ tiles. In this case $B$ consists of quadrilaterals made up of pairs of $\Delta^{(3)}$ triangles  glued along the type 1 side; or
				\item [(ii)] $B$ contains a core $C$ consisting of  $n \ge 1$ $\Delta^{(1)}$ tiles,   with one $\Delta^{(3)}$ tile glued to each of the sides of $C$. In this case, all the ears of $B$ are $\Delta^{(3)}$ tiles and there are $n+2$ of them, and the weights of the vertices alternate between  1 and  $m \ge 3$.  
			\end{enumerate}
		\end{enumerate} 
		
		We call blocks of the first type $0$-blocks, and blocks of the second type $n$-blocks where $n$ is the number of tiles in the core $C$. In particular, a $1$-block is just $J_{A}$ and our aim will be to show that other types of blocks cannot occur. Note that by definition, all the sides of a block are of type $3$.
		
		We now show that an arithmetic $J$ can be decomposed into blocks defined above. By proposition \ref{p:arithemtic13}, all the sides of $J$ are of type 3 and by proposition \ref{p:nonarith13}, the $J$ widths of all vertices of $J$ are congruent to 0 mod 3. Also, by definition,  $J$ contains at least one $\Delta^{(1)}$ tile. We thus have a unique $n$-block $B_1$ (where $n \ge 1$) which contains this tile and $J \setminus B_1$ consists of a finite number of polygons whose sides are again all of type 3. Proceeding inductively, we see that $J$ can be decomposed into a finite number of blocks $B_1, B_2, \ldots B_n$. 
		
		Now for a block $B$, if $v$ is a vertex of $B$ with $J$-width $\JW(v)$, we say that $v$ is bad if $\JW(v) \not \equiv 0 \mod 3$ and good if $\JW(v) \equiv 0 \mod 3$. Then a $1$-block has only good vertices since all its six vertices have $J$-width $3$.
		
		 We claim that a $n$-block with $n \neq 1$ has at least two non-adjacent bad vertices. 
		 
		 This is certainly true for the $0$-blocks since each quadrilateral consisting of two $\Delta^{(2)}$ tiles glued along their one sides have two bad vertices which are not adjacent. Gluing a finite number of these quadrilaterals along matching sides will not reduce this number of bad vertices. Indeed, each gluing fixes at most one bad vertex but introduces at least one new bad vertex which will not be adjacent to the unfixed bad vertex. 
		 
		 For $n$-blocks $B$ where $n \ge 2$, we note that by Lemma \ref{l:badvertices}, the core $C$ has at least two vertices $u$ and $v$ with $J$-width equal to $2$ or $3$. Then the $J$-width of the corresponding vertices on $B$ will be $4$ or $5$ (we add 2 to the original widths coming from the $J$-widths of the ears). It follows that $B$ has at least two bad vertices which are not adjacent. 
		 
		 Hence,  if $J$ has an $n$-block $B$ with $n \neq 1$, then $B$ has at least two non-adjacent  bad vertices. If we glue a $J_A$ block to $B$, then the two bad vertices persist and remain non-adjacent, if we glue an $n$-block with $n \neq 1$ to $B$, then  we fix at most one bad vertex of $B$ but introduce at least one new bad vertex which will not be adjacent to the unfixed bad vertex of $B$. thus, adjoining another block to $B$ results in a jigsaw with at least two bad vertices which are not adjacent. The same argument holds when we glue another block to this new jigsaw. Hence, after a finite number of steps, we will always end up with a jigsaw $J$ which has at least two non-adjacent bad vertices. By proposition \ref{p:nonarith13},   $J$ cannot be arithmetic. Hence, an arithmetic $J$ consists only of $J_A$ blocks and the result follows.
		
		\end{proof}


\end{proof}

\section{Commensurability classes }\label{s:commensurabilty}
By Margulis's result \cite{Mar}, if $\Gamma < \PSLtwoR$ is non-arithmetic, then   $\hbox{Comm}(\Gamma)$ is the unique minimal element  in the commensurability class $[\Gamma]$, where
\[\hbox{Comm}(\Gamma)=\{g \in \PSLtwoR ~|~ \Gamma ~~\hbox{and} ~~g\Gamma g^{-1} ~~\hbox{are commensurable}\},\] and $\Gamma < \hbox{Comm}(\Gamma)$ with finite index. We will show that for all $\SSS(1,2)$ or $\SSS(1,3)$ jigsaws $J$ of signature $(r,1)$,
 $$\hbox{Comm}(\Gamma_J)=\Gamma_J$$ which will be enough to show that there are infinitely many commensurability classes of pseudomodular jigsaw groups. 
 
 If $\Gamma_J$ is a pseudomodular hyperbolic jigsaw group, let  $\Gamma'_J:=\hbox{Comm}(\Gamma_J)$ be the minimal element in the commensurability class of $[\Gamma_J]$, $S'_J$  the associated surface and $C$ the maximal horocycle on $S'_J$. Let $\widetilde{C}$ be the lift of  $C$  to $\HH$. By construction $\widetilde{C}$ is invariant under the action of $\Gamma'_J$. Also, since $S_J$ has only one cusp, the lift of the maximal horocyle of $S_J$ to $\HH$ is also $\widetilde{C}$. We have:

\begin{prop}\label{p:commsurability12}
	There exists infinitely many commensurability classes of pseudomodular $\SSS(1,2)$-jigsaw groups.
\end{prop}
	\begin{prop}\label{p:commsurability123}
		 There exists infinitely many commensurability classes of pseudomodular $\SSS(1,3)$-jigsaw groups and infinitely many commensurability classes of $\SSS(1,3)$-jigsaw groups with specials (so are neither arithmetic nor pseudomodular).
	
\end{prop}

\begin{proof}(Propositions \ref{p:commsurability12} and \ref{p:commsurability123})
First consider an $\SSS(1,2)$-jigsaw $J$ of signature $(r,1)$. The fundamental interval has (euclidean) length $3r+4$ and $T^{3r+4} \in \Gamma_J$. The maximal horocycle $C$   on $S_J$ has (hyperbolic) length $(3r+4)/\sqrt{2}$ and intersects itself at the fixed points of the sides labeled by $2$ and $1/2$. It decomposes into 2 pieces, a short piece of hyperbolic length $\sqrt{2}$ (the length of the horocyclic segment joining $a+i\sqrt{2}$ to $a+2+i\sqrt{2}$), and a long piece of hyperbolic length $(3r+2)/\sqrt{2}$, where $(\infty, a)$ and $(\infty, a+2)$ are the lifts of two successive vertical sides of $\TTT$ of type $2$. Then $\widetilde{C}$ consists of the horizontal horocyle at height $\sqrt{2}$, ``large horocycles'' of diameter $\sqrt{2}/2$ based at the points $a+n(3r+4), a+2+n(3r+4) \in \RR$, where $n \in \ZZ$, which are tangent to the horizontal horocyle, and smaller horocycles based at the other cusps which are disjoint from the horizontal horocycle. In other words, the horizontal lift of $C$ has exactly two tangency points in each fundamental interval dividing it into two pieces of length $\sqrt{2}$ and $(3r+2)/\sqrt{2}$ respectively.
 We claim that $\Gamma_J=\Gamma'_J$. If not,  since $\Gamma_J$ has only one cusp,  $\widetilde{C}$ is invariant under a horizontal translation by $k$ where $0<k<3r+4$ but this is impossible as such a translation cannot preserve the pattern of tangency points at the horizontal lift of $C$.
   It follows that each $\SSS(1,2)$-jigsaw group $\Gamma_J$ of signature $(r,1)$ is the unique minimal element in its conjugacy class and that two non-isometric jigsaws with only one $\Deltatwo$ tile gives rise to pseudomodular groups in different commensurability classes. This  gives infinitely many commensurability classes of pseudomodular groups. 

The same argument with minor modifications works for $\SSS(1,3)$-jigsaws. We only need to note that if $J$ is a $\SSS(1,3)$ jigsaw of signature $(r,1)$, then it must be pseudomodular since it does not contain a geodesic with cutting sequence $\overline{1,3,1,1/3}$, and it has a side of type 1.  To show the existence of infinitely many commensurability classes of non-pseudomodular $\SSS(1,3)$-jigsaw groups, we only need to start with a jigsaw $J'$ consisting of three $\Delta^{(3)}$ triangles such that there exists a geodesic $\gamma$ on $S_{J'}$ with $W$ cutting sequence $\overline{1,1/3,1/3}$, and $J'$ has a side  which does not intersect $\gamma$ with label 1, see figure \ref{fig:nonpseudomodular}. We then glue another jigsaw $J_n$ to $J'$ along this side, where $J_n$ consists of $n$ $\Delta^{(1)}$ tiles. The jigsaw groups constructed will all be non-arithmetic and non-pseudomodular. Again, the resulting jigsaw groups coming from $J_n'=J'\cup J_n$  lie in infinitely many commensurability classes, by a similar argument.
\end{proof}

\begin{figure}[hbt]\centering{\includegraphics[height=6cm]{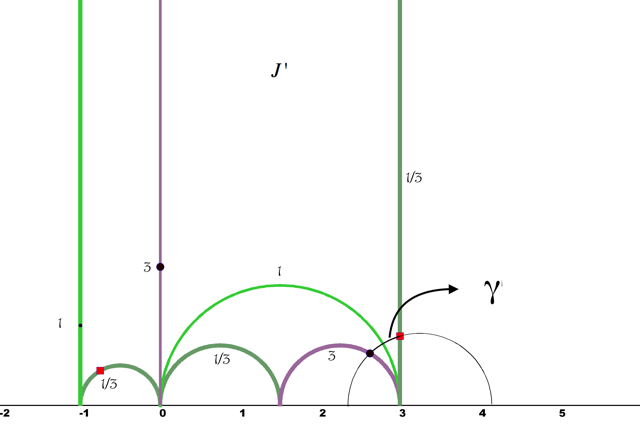}}
	\caption{The hyperbolic jigsaw $J'$ consisting of three $\Delta^{(3)}$ tiles and the geodesic arc on $\Gamma_{J'}$ which has cutting sequence $\overline{3,1,1/3,1}$.}  
	\label{fig:nonpseudomodular} 
\end{figure}

\section{Conclusion}\label{s:conclusion} Propositions \ref{cuspset12}, \ref{p:cuspset13}, \ref{p:nonarith12}, \ref{p:nonarith13}, \ref{p:commsurability12} and \ref{p:commsurability123} now imply  Theorems \ref{thm:S12} and \ref{thm:S13} which in turn both imply Theorem \ref{thm:Main}. One sees that integral hyperbolic jigsaws provide a rich source of pseudomodular groups. We will be exploring this in a subsequent paper \cite{LTV} where we will show that for any integral hyperbolic jigsaw set $\SSS$, there are infinitely many commensurability classes of pseudomodular $\SSS$-jigsaw groups. We will also be exploring the related pseudo-Euclidean algorithm as well as the generalized continued fraction expansion for these groups as well as for the integral Weierstrass groups. Many questions remain, we conclude with a partial list:
\begin{enumerate}
	\item Can one give sharp invariants that will distinguish the commensurability classes of the pseudomodular groups, for example, by using the associated labeled graph $G_J$ of $J$. In particular, we may define an $\SSS$-jigsaw to be prime if it cannot be decomposed into  pieces all of which are isometric. Is it then true that $\hbox{Comm}(\Gamma_J)=\Gamma_J$ if and only if $J$ is prime? 
	
	\item Can one find integral jigsaw groups where the set of fixed points of the hyperbolics is exactly the set of quadratic irrationals? Or where some quadratic irrational does not lie in this set? More generally, what can be said about the fixed point set of the hyperbolics for a given jigsaw group.
	
	\item Are there only finitely many pseudomodular Weierstrass groups $\Gamma(1,1/n,n)$ where $n \in \NN$? Which Weierstrass groups are arithmetic?
	\item What integral jigsaw groups can be arithmetic? It seems likely that the arithmetic integral jigsaw groups only come from a small finite collection of jigsaw sets. If so, can we determine this collection completely?
	\item One can consider jigsaw tiles $\Delta(k_1,k_2,k_3)$ where $k_i \in \QQ[i]$ which results in groups $\Gamma_J \in \PSLtwoC$. What can we say about these groups? Is it possible to construct a jigsaw with cusp set $\QQ[i]\cup \{\infty\}$?
	 
\end{enumerate}

\end{document}